\documentclass[11pt]{amsart}
\usepackage{amsopn}
\usepackage{amssymb, amscd}

\theoremstyle{plain}
\newtheorem{theorem}{Theorem}[section]
\newtheorem{proposition}[theorem]{Proposition}

\newtheorem{lemma}[theorem]{Lemma}

\theoremstyle{definition}

\theoremstyle{remark}
\newtheorem{remark}[theorem]{Remark}

\newtheorem{example}[theorem]{Example}

\newcommand{\nc}{\newcommand}

\nc{\fg}{\mathfrak{f} } \nc{\vg}{\mathfrak{v} } \nc{\wg}{\mathfrak{w} }
\nc{\zg}{\mathfrak{z} } \nc{\ngo}{\mathfrak{n} } \nc{\kg}{\mathfrak{k} }
\nc{\mg}{\mathfrak{m} } \nc{\bg}{\mathfrak{b} } \nc{\ggo}{\mathfrak{g} }
\nc{\ggob}{\overline{\mathfrak{g}} } \nc{\sog}{\mathfrak{so} }
\nc{\sug}{\mathfrak{su} } \nc{\spg}{\mathfrak{sp} } \nc{\slg}{\mathfrak{sl} }
\nc{\glg}{\mathfrak{gl} } \nc{\cg}{\mathfrak{c} } \nc{\rg}{\mathfrak{r} }
\nc{\hg}{\mathfrak{h} } \nc{\tg}{\mathfrak{t} } \nc{\ug}{\mathfrak{u} }
\nc{\dg}{\mathfrak{d} } \nc{\ag}{\mathfrak{a} } \nc{\pg}{\mathfrak{p} }
\nc{\sg}{\mathfrak{s} } \nc{\pca}{\mathcal{P}} \nc{\nca}{\mathcal{N}}
\nc{\lca}{\mathcal{L}} \nc{\oca}{\mathcal{O}} \nc{\mca}{\mathcal{M}}
\nc{\tca}{\mathcal{T}} \nc{\aca}{\mathcal{A}} \nc{\cca}{\mathcal{C}}
\nc{\gca}{\mathcal{G}} \nc{\sca}{\mathcal{S}} \nc{\hca}{\mathcal{H}}
\nc{\bca}{\mathcal{B}} \nc{\dca}{\mathcal{D}} \nc{\val}{\operatorname{val}}

\nc{\vp}{\varphi} \nc{\ddt}{\tfrac{{\rm d}}{{\rm d}t}} \nc{\im}{\mathtt{i}}
 \nc{\dpar}{\tfrac{\partial}{\partial t}}

\nc{\SO}{\mathrm{SO}} \nc{\Spe}{\mathrm{Sp}} \nc{\Sl}{\mathrm{SL}}
\nc{\SU}{\mathrm{SU}} \nc{\Or}{\mathrm{O}} \nc{\U}{\mathrm{U}} \nc{\Gl}{\mathrm{GL}}
\nc{\Se}{\mathrm{S}} \nc{\Cl}{\mathrm{Cl}} \nc{\Spein}{\mathrm{Spin}}
\nc{\Pin}{\mathrm{Pin}} \nc{\G}{\mathrm{GL}_n(\RR)} \nc{\g}{\mathfrak{gl}_n(\RR)}

\nc{\RR}{{\Bbb R}} \nc{\HH}{{\Bbb H}} \nc{\CC}{{\Bbb C}} \nc{\ZZ}{{\Bbb Z}}
\nc{\FF}{{\Bbb F}} \nc{\NN}{{\Bbb N}} \nc{\QQ}{{\Bbb Q}} \nc{\PP}{{\Bbb P}}

\nc{\vs}{\vspace{.2cm}} \nc{\vsp}{\vspace{1cm}} \nc{\ip}{\langle\cdot,\cdot\rangle}
\nc{\ipp}{(\cdot,\cdot)} \nc{\la}{\langle} \nc{\ra}{\rangle} \nc{\unm}{\tfrac{1}{2}}
\nc{\unc}{\tfrac{1}{4}} \nc{\und}{\tfrac{1}{16}} \nc{\no}{\vs\noindent}
\nc{\lam}{\Lambda^2(\RR^n)^*\otimes\RR^n} \nc{\tangz}{{\rm T}^{\rm Zar}}
\nc{\nor}{{\sf n}}  \nc{\mum}{/\!\!/} \nc{\kir}{/\!\!/\!\!/}
\nc{\Ri}{\tfrac{4\Ric_{\mu}}{||\mu||^2}} \nc{\ds}{\displaystyle}
\nc{\ben}{\begin{enumerate}} \nc{\een}{\end{enumerate}} \nc{\f}{\frac}
\nc{\lb}{[\cdot,\cdot]} \nc{\isn}{\tfrac{1}{||v||^2}}

\nc{\He}{\operatorname{Hess}} \nc{\ad}{\operatorname{ad}}
\nc{\Ad}{\operatorname{Ad}} \nc{\rank}{\operatorname{rank}}
\nc{\Irr}{\operatorname{Irr}} \nc{\End}{\operatorname{End}}
\nc{\Aut}{\operatorname{Aut}} \nc{\Inn}{\operatorname{Inn}}
\nc{\Der}{\operatorname{Der}} \nc{\Ker}{\operatorname{Ker}}
\nc{\Iso}{\operatorname{I}} \nc{\Diff}{\operatorname{D}} \nc{\Lie}{\operatorname{L}}
\nc{\tr}{\operatorname{tr}} \nc{\dif}{\operatorname{d}}
\nc{\sen}{\operatorname{sen}} \nc{\modu}{\operatorname{mod}}
\nc{\ricci}{\operatorname{Rc}} \nc{\Ricci}{\operatorname{Ric}}
\nc{\sym}{\operatorname{sym}} \nc{\symac}{\operatorname{sym^{ac}}}
\nc{\symc}{\operatorname{sym^{c}}} \nc{\scalar}{\operatorname{sc}}
\nc{\grad}{\operatorname{grad}}
\nc{\ricciac}{\operatorname{ric^{ac}}} \nc{\riccic}{\operatorname{ric^{c}}}
\nc{\riccig}{\operatorname{ric^{\gamma}}} \nc{\Rin}{\operatorname{M}}
\nc{\Le}{\operatorname{L}} \nc{\tang}{\operatorname{T}}
\nc{\level}{\operatorname{level}} \nc{\rad}{\operatorname{r}}
\nc{\abel}{\operatorname{ab}} \nc{\CH}{\operatorname{CH}}
\nc{\mcc}{\operatorname{mcc}} \nc{\Adj}{\operatorname{Adj}}
\nc{\diag}{\operatorname{Diag}}

\title{On the diagonalization of the Ricci flow on Lie groups}

\author{Jorge Lauret} \author{Cynthia Will}

\thanks{This research was partially supported by grants from  CONICET (Argentina)
and SeCyT (Universidad Nacional de C\'ordoba)}

\address{FaMAF and CIEM, Universidad Nacional de C\'ordoba, C\'ordoba, Argentina}
\email{lauret@famaf.unc.edu.ar, cwill@famaf.unc.edu.ar}

\begin{document}

\maketitle

\begin{abstract}
The main purpose of this note is to prove that any basis of a nilpotent Lie algebra for which all diagonal left-invariant metrics have diagonal Ricci tensor  necessarily produce quite a simple set of structural constants; namely, the bracket of any pair of elements of the basis must be a multiple of some of them and only the bracket of disjoint pairs can be nonzero multiples of the same element.  Some applications to the Ricci flow of left-invariant metrics on Lie groups concerning diagonalization are also given.
\end{abstract}

\section{Introduction}\label{intro}

The classification of all possible signatures for the Ricci curvature of left-invariant metrics on $3$-dimensional unimodular Lie groups obtained by Milnor in \cite{Mln} relies on the following key fact: any such Lie algebra admits a basis such that the corresponding diagonal metrics represent all left-invariant metrics up to isometry and moreover, their Ricci tensors are all diagonal as well.  This of course has also been crucial in the study of the Ricci flow and Ricci solitons for these metrics in \cite{IsnJck,KnpMcl,Glc,CaoSlf,GlcPyn}, which actually cover most of $3$-dimensional geometries from the Geometrization Conjecture.  Indeed, if a basis $\{ X_1,\dots,X_n\}$ of a Lie algebra is {\it stably Ricci-diagonal} in the sense that any diagonal left-invariant metric (i.e. $\la X_i,X_j\ra=0$ for all $i\ne j$) has diagonal Ricci tensor, then the set of diagonal metrics is invariant under the Ricci flow and hence the qualitative study of the solutions reduces considerably.  These distinguished bases were introduced and named by Payne in \cite{Pyn}, where a deep study of the qualitative behavior of the Ricci flow for those left-invariant metrics on nilpotent Lie groups admitting such a basis has been done.

It is therefore natural to ask which Lie algebras admit a stably Ricci-diagonal basis, or even farther, which left-invariant metrics admit an orthonormal stably Ricci-diagonal basis.  For instance, the study in \cite{IsnJckLu} of the Ricci flow on homogeneous $4$-manifolds could not be carried out for the totality of left-invariant metrics (even on the $3$-step nilpotent Lie group) due to the lack of these special bases (see also \cite{Ltt}).

From a more algebraic point of view, there is a condition on the basis of a Lie algebra based on the simplicity of the corresponding set of structural constants.  Namely, a basis $\{ X_1,\dots,X_n\}$ of a Lie algebra is said to be {\it nice} if $[X_i,X_j]$ is always a scalar multiple of some element in the basis and two different brackets $[X_i,X_j]$, $[X_r,X_s]$ can be a nonzero multiple of the same $X_k$ only if $\{ i,j\}$ and $\{ r,s\}$ are disjoint.  As far as we know, this concept first appeared in the literature in \cite[Lemma 3.9]{einsteinsolv}, and has been widely used in the study of {\it nilsolitons} (i.e. Ricci soliton nilmanifolds).  For instance, Nikolayevsky obtained a very useful and simple criterium to decide whether a given nilpotent Lie algebra with a nice basis admits a nilsoliton or not (see \cite[Theorem 3]{Nkl2}).  An updated overview on the existence of these bases in the nilpotent case is given in Section \ref{exist}.

Let $(N,g)$ be a nilpotent Lie group endowed with a left-invariant metric (i.e. a {\it nilmanifold} for short).  If $(\ngo,\ip)$ is the corresponding {\it metric Lie algebra} (i.e. the value $g(e)$ of the metric on the tangent space at the identity element $\tang_eN=\ngo$ is $\ip$), then the Ricci tensor is given by
\begin{equation}\label{ricci}
\ricci(X,Y)= -\unm\sum\langle [X,X_i],X_j\rangle \langle [Y,X_i],X_j\rangle +\unc\sum\langle [X_i,X_j],X\rangle \langle [X_i,X_j],Y \rangle,
\end{equation}
for any orthonormal basis $\{ X_1,\dots,X_n\}$ of $(\ngo,\ip)$.

It follows from (\ref{ricci}) that any nice basis is stably Ricci-diagonal.  Indeed, the basis $\{ X_1,\dots,X_n\}$ is orthogonal relative to any diagonal metric $\ip$ and clearly, each summand in the formula for $\ricci(X_r,X_s)$ vanishes when $r\ne s$.  The main object of this note is to prove that the converse assertion also holds.

\begin{theorem}\label{Thmintro}
A basis of a nilpotent Lie algebra is stably Ricci-diagonal if and only if it is nice.
\end{theorem}

A purely geometric condition on a basis of a nilpotent Lie algebra is therefore characterized as a neat algebraic condition.  The proof is based on the fact that the Ricci operator is precisely the moment map for the $\Gl_n(\RR)$-representation where the Lie brackets live.  In the non-nilpotent case, both implications in the theorem are not longer true (see Examples \ref{solv3}, \ref{solv4}, \ref{sl3}).

Back to the Ricci flow, what Theorem \ref{Thmintro} is saying is that, in the case of a nilmanifold $(N,g_0)$, to ask for a stably-Ricci diagonal basis in order to get a diagonal Ricci flow solution $g(t)$ starting at $g_0$ is quite expensive (see \cite{nilricciflow} for a study of the Ricci flow on nilmanifolds without the use of these special bases).  We prove for instance that any {\it algebraic soliton} $(N,g_0)$ (i.e. $\Ricci(g_0)\in\RR I\oplus\Der(\ngo)$, a condition which easily implies that $(N,g_0)$ is indeed a Ricci soliton), where $N$ is any (non-necessarily nilpotent) Lie group, gives rise to a diagonal Ricci flow solution (see Example \ref{RS}).

On the other hand, already in dimension $4$, there are nilmanifolds whose Ricci flow is not diagonal with respect to any orthonormal basis (see Example \ref{nodiag}).

\section{On the existence of a nice basis}\label{exist}

Let $\ngo$ be a nilpotent Lie algebra. A basis $\{X_1, \dots, X_n \}$ of $\ngo$ is called {\it nice} if the structural constants given by $[X_i,X_j]=\sum c_{ij}^k X_k$ satisfy that
\begin{equation}\label{nicecond}
\begin{array}{@{\bullet\;\;}l}
  \text{for all } i,j \text{ there exists at most one } k \text{ such that } c_{ij}^k \ne 0, \\
  \text{for all } i,k  \text{ there exists at most one } j \text{ such that } c_{ij}^k \ne 0.
\end{array}
\end{equation}

$\ngo$ is said to be of {\it type} $(n_1,...,n_r)$ if $n_i=\dim{C^{i-1}(\ngo)/C^i(\ngo)}$, where $C^i(\ngo)$ is the central descendent series of $\ngo$ defined by $C^0(\ngo)=\ngo$, $C^i(\ngo)=[\ngo,C^{i-1}(\ngo)]$.  One can always take a direct sum decomposition $\ngo=\ngo_1\oplus...\oplus\ngo_r$ such that $C^i(\ngo)=\ngo_{i+1}\oplus...\oplus\ngo_r$ for all $i$, and so $\dim{\ngo_i}=n_i$.

If the type of $\ngo$ is $(n_1,...,n_r)$, then we can reorder any nice basis so that the first $n_1$ elements give a basis for $\ngo_1$, the following $n_2$ give a basis for $\ngo_2$ and so on. This easily follows from the fact that a nonzero bracket $[X_i,X_j]$ must be a multiple of some $X_k$.

Concerning existence, from the explicit bases exhibited in the classification lists available in the literature, we have that any nilpotent Lie algebra of dimension $\le 5$ as well as any filiform $\NN$-graded Lie algebra admits a nice basis (see e.g. \cite{Nkl1}).  Also, by using the classification of $6$-dimensional nilpotent Lie algebras given for example in \cite{dGr}, it is easy to check that all of the $34$ algebras in the list are written in a nice basis with the only exception of $L_{6,11}$, which is denoted by $\mu_{11}$ in \cite{Wll}.  We now prove that indeed, this algebra can not have a nice basis (this fact has independently been mentioned in \cite{Frn} after Definition 2.7, joint with an idea of the proof).

\begin{proposition}\label{ex6}
The $6$-dimensional $4$-step nilpotent Lie algebra of type $(3,1,1,1)$ defined by
$$
\begin{array}{lcl}
[X_1,X_2]=X_4, & [X_1,X_4]= X_5,& [X_1,X_5]= [X_2,X_3]=[X_2,X_4]= X_6,
\end{array}
$$
does not admit any nice basis.
\end{proposition}

\begin{proof}
Let $\ngo$ be a $4$-step nilpotent Lie algebra of type $(3,1,1,1)$ with no abelian factor. Assume that $\{X_1,\dots,X_6\}$ is a nice bases of $\ngo$ so that $\ngo_1= \text{span}\{ X_1,X_2,X_3\}$, $\ngo_2=\RR X_4,$  $\ngo_3=\RR X_5$ and  $\ngo_4=\RR X_6$.  Thus $X_6$ is in the center of $\ngo$ and there exist $i,j,k,l\in\{ 1,2,3\}$ such that
\begin{equation}\label{3111}
[X_i,X_j]=aX_4, \quad [X_k,X_4]=bX_5, \quad [X_l,X_5]=cX_6, \qquad a,b,c \ne 0.
\end{equation}
Note that since the basis is nice, the brackets
$[X_1,X_2]$, $[X_1,X_3]$, $[X_2,X_3]$ if nonzero, they are linearly independent. We therefore have three cases to consider:
$$
d:=\dim{\la [X_1,X_2],\; [X_1,X_3],\; [X_2,X_3]\ra} =1,\;2 \text{ or } 3,
$$
and in any case we can assume that $[X_1,X_2]=X_4$.

The simplest one corresponds to $d=1$, where by using Jacobi identities it is easy to see that $\RR X_3$ is an abelian factor.
If $d=3$, that is,
$$
[X_1,X_2]=X_4,\quad [X_1,X_3]=X_5, \quad [X_2,X_3]=X_6,
$$
then we also must have
$$
[X_2,X_4]= b' X_5, \quad [X_1,X_5]=c' X_6,
$$
for some $b',c'\in\RR$.  The Jacobi identities now imply that $b'c'=0$ and therefore (\ref{3111}) would not hold.

Finally, if $d=2$, we may assume that
$$
[X_1,X_2]=X_4, \quad [X_2,X_3]= \alpha X_5+\beta X_6,
$$
where either $\alpha \ne 0$ or $\beta \ne 0$.  If $\beta=0$, by assuming that $\alpha =1$, we obtain that $k=1$ and $[X_3,X_4]=-[X_1,X_5]$.  Thus the non trivial Lie brackets are
$$
\begin{array}{lcl}
[X_1,X_2]=X_4, & [X_2,X_3]= X_5,& [X_3,X_4]= X_6,\\
& [X_1,X_4]= b X_5, &[X_1,X_5]= -X_6,
\end{array}
$$
for some $b \ne 0$.  A change of bases given by
$$
\{X_1, X_2, {\epsilon} X_3, \epsilon^{-1} X_4, {\epsilon} X_5, {\epsilon} X_6\},
$$
where $\epsilon = |b|^{1/2},$ shows that $\ngo$ is isomorphic to $L_{6,13}$.
A similar argument shows that if $\alpha =0$ we get $L_{6,12}$.

We have then showed that any type-$(3,1,1,1)$ $4$-step nilpotent Lie algebra with no abelian factor that admits a nice basis must be isomorphic to either $L_{6,12}$ or $L_{6,13}$, and thus $L_{6,11}$ can not admit a nice basis. We note that the fact that these three algebras are pairwise non-isomorphic also follows by using that they admit nilsolitons of different eigenvalue types (compare with $\mu_9$, $\mu_{10}$ and $\mu_{11}$ in \cite[Tables 1,3]{Wll}).
\end{proof}

A distinguished class of nilpotent Lie algebras admitting a nice basis is the given by nilradicals of Borel subalgebras of any semisimple Lie algebra.

On the other hand, it is proved in \cite[Example 4]{Nkl2} that the free $3$-step nilpotent Lie algebra in $3$ generators (which is of type $(3,3,8)$) does not admit a nice basis and, by a dimensional argument, it is also shown that there exist infinitely many $2$-step nilpotent Lie algebras with no  nice basis for any type $(p,q)$ such that
$$
\unm\min\{q(q-1),pq\} +q^2+p^2-1< \unm pq(q-1).
$$
This condition holds, for example, for any $q-1 \ge p \ge 6$ and the $13$-dimensional case $(6,7)$ is the lowest one.  In dimension $7,$ by using the classification given in \cite{dGr}, one can check that from the 117 algebras and $6$ curves of $7$-dimensional nilpotent Lie algebras all but $20$ algebras and 1 curve are written in a nice basis (see also \cite{Frn}).  We do not know whether these exceptions admit or not a nice basis.

We can also mention the following sufficient condition: any nilpotent Lie algebra admitting a {\it simple} derivation (i.e. diagonalizable over $\RR$ and with all its eigenvalues of multiplicity one) has a nice basis. Indeed, any basis of eigenvectors $\{X_1,\dots,X_n\}$ of $D$ is automatically nice since $[X_i,X_j]$ is either $0$ or also an eigenvector of $D$ and if $[X_i,X_j],[X_r,X_s]\in\RR X_k$, then the (pairwise different) eigenvalues satisfy  $d_i+d_j=d_r+d_s=d_k$, which implies that $\{ i,j\}$ and $\{ r,s\}$ are either equal or disjoint.  We note that this is not a necessary condition.  One may check for example in \cite[Tables 4-6]{solvsol}, that the $6$-dimensional algebras $\mu_8$ and $\mu_{18}$ do not admit any simple derivation and they both however admit a nice bases.

\section{Technical preliminaries}\label{prelim}

Let us consider the space of all skew-symmetric algebras of dimension $n$, which is
parameterized by the vector space
$$
V=\lam=\{\mu:\RR^n\times\RR^n\longrightarrow\RR^n : \mu\; \mbox{bilinear and
skew-symmetric}\}.
$$
There is a natural linear
action of $\G$ on $V$ defined by
\begin{equation}\label{action}
A\cdot\mu(X,Y)=A\mu(A^{-1}X,A^{-1}Y), \qquad X,Y\in\RR^n, \quad A\in\G,\quad \mu\in V,
\end{equation}
and the corresponding representation of the Lie algebra $\g$
of $\G$ on $V$ is given by
\begin{equation}\label{actiong}
\pi(\alpha)\mu=\alpha\mu(\cdot,\cdot)-\mu(\alpha\cdot,\cdot)-\mu(\cdot,\alpha\cdot),
\qquad \alpha\in\g,\quad\mu\in V.
\end{equation}
The canonical inner product $\ip$ on $\RR^n$ determines inner
products on $V$ and $\g$, both also denoted by $\ip$, as follows:
\begin{equation}\label{innV}
\la\mu,\lambda\ra= \sum\limits_{ijk}\la\mu(e_i,e_j),e_k\ra\la\lambda(e_i,e_j),e_k\ra, \qquad \la \alpha,\beta\ra=\tr{\alpha \beta^{\mathrm t}},
\end{equation}
where $\{e_1,...,e_n\}$ denote the canonical basis of $\RR^n$ and $\beta^t$ the transpose with respect to $\ip$.  We note that $\pi(\alpha)^t=\pi(\alpha^t)$ and $(\ad{\alpha})^t=\ad{\alpha^t}$ for
any $\alpha\in\g$, due to the choice of canonical inner products everywhere.

We use $\g=\sog(n)\oplus\sym(n)$ as a Cartan decomposition of $\g,$ where $\sog(n)$ and $\sym(n)$ denote the subspaces of skew-symmetric and
symmetric matrices, respectively.  The set $\ag$ of all diagonal $n\times n$ matrices is a maximal abelian subalgebra of $\sym(n)$ and therefore determines a system of roots $\Delta\subset \ag$.  Let $\Phi$ denote the set of positive roots, which is given by
\begin{equation}\label{sr}
\Phi=\{ E_{ll}-E_{mm}\in\ag, \;\; l>m\},
\end{equation}
where $E_{rs}$ denotes the matrix whose only nonzero
coefficient is $1$ at entry $rs$.  We have the root space decomposition $\g= \ag \oplus\bigoplus\limits_{\lambda\in \Delta} \ggo_\lambda$, where for each $\lambda \in \Delta$,
$$
\ggo_{\lambda}=\{X \in \g:\; [\alpha, X]=\la\alpha,\lambda\ra X, \quad \forall \alpha \in \ag\}.
$$
If $\{
e_1',...,e_n'\}$ is the basis of $(\RR^n)^*$ dual to the canonical basis $\{
e_1,...,e_n\}$, then
$$
\{ v_{ijk}=(e_i'\wedge e_j')\otimes e_k : 1\leq i<j\leq n, \; 1\leq k\leq n\}
$$
is a basis of weight vectors of $V$ for the representation (\ref{actiong}), where $v_{ijk}$
is actually the bilinear form on $\RR^n$ defined by
$v_{ijk}(e_i,e_j)=-v_{ijk}(e_j,e_i)=e_k$ and zero otherwise.  The corresponding
weights $\alpha_{ij}^k\in\ag$, $i<j$, are given by
\begin{equation}\label{alfas}
\pi(\alpha)v_{ijk}=(a_k-a_i-a_j)v_{ijk}=\la\alpha,\alpha_{ij}^k\ra v_{ijk},
\quad\forall\alpha=\left[\begin{array}{ccc} a_1&&\\ &\ddots&\\ &&a_n
\end{array}\right]\in\ag,
\end{equation}
where $\alpha_{ij}^k:=E_{kk}-E_{ii}-E_{jj}$.

The representation $(V,\pi)$ of $\glg_n(\RR)$ is multiplicity free, i.e. the weight spaces have all dimension one.  An additional special feature of this representation is that the sum of two weights $\alpha_{ij}^k + \alpha_{rs}^{t}$ is never zero.

\section{Proof of Theorem \ref{Thmintro}}

By fixing any basis of an $n$-dimensional nilpotent Lie algebra $\ngo$, we can identify its underlying vector space $\ngo$ with $\RR^n$
and view its bracket $\lb$ as an element of $V=\lam$ (see Section \ref{prelim} for all definitions and notation used in what follows).  The structural constants $c_{ij}^k$ are therefore given by
$$
[e_i,e_j]=\sum_{k}c_{ij}^ke_k,
\qquad\mbox{or}\qquad \lb=\sum_{k;\, i<j}c_{ij}^kv_{ijk}.
$$
We note that
\begin{quote}
$\ngo$ admits a nice basis if and only if the canonical basis $\{ e_1,\dots,e_n\}$ is nice for some $A\cdot\lb\in V$ with $A \in \G$,
\end{quote}
as this is precisely the `change of basis' action.  Let us first show that the nice condition on a bracket $\lb$, in the sense that the canonical basis is nice for $\ngo$, can be written in terms of the set $\Phi$ of positive roots of $\glg_n(\RR)$ (see (\ref{sr})) and the weights of the representation $(V,\pi)$ (see (\ref{alfas})).

\begin{lemma}\label{nicemu}
The canonical basis $\{ e_1,\dots,e_n\}$ is nice for $\ngo$ if and only if
$$
\alpha_{ij}^k - \alpha_{rs}^t \notin \Phi, \quad \text{ for any } \quad  c_{ij}^k, c_{rs}^{t}\ne 0.
$$
\end{lemma}

\begin{proof}
According to (\ref{nicecond}), $\{ e_1,\dots,e_n\}$ is nice for $\ngo$ if and only if
\begin{equation}\label{II}
\sharp \{k: c_{ij}^k \ne 0\} \le 1 \qquad \mbox{  and  } \qquad \sharp \{j: c_{ij}^k \ne 0\} \le 1.
\end{equation}
If $\alpha_{ij}^k - \alpha_{rs}^t \in \Phi$ for some $c_{ij}^k, c_{rs}^{t}\ne 0$, then
\begin{equation}\label{resta}
E_{kk}-E_{ii}-E_{jj}- E_{tt}+E_{rr}+E_{ss}= E_{ll}-E_{mm},
\end{equation}
for some $m<l$.
By using the fact that $\ngo$ is nilpotent and therefore $c_{ij}^k \ne 0$ implies that $i,j \ne k,$ it is easy to check that this cancelation can only be possible when either $\{i,j\}=\{r,s\}$ and $t<k$, or $k=t$ and $\sharp (\{i,j\}\cap \{r,s\})=1$. In any case, one of the conditions in (\ref{II}) will not hold.  In other words, $\alpha_{ij}^k - \alpha_{rs}^t \in \Phi$ for some $c_{ij}^k, c_{rs}^{t}\ne 0$ if and only if $\{ e_1,\dots,e_n\}$ is not a nice basis for $\ngo$, as was to be shown.
\end{proof}

We note that each $\mu\in V$ satisfying the Jacobi identity determines a Lie group $N_\mu$ (the simply connected Lie group with Lie algebra $(\RR^n,\mu)$), which can be endowed with the left invariant metric defined by the canonical inner product $\ip$ (fixed).  Let us denote by $\ricci_\mu$ the Ricci tensor of such a metric, and by $\Ricci_\mu$ its Ricci operator.  The action of $\G$ on $V$ given by (\ref{action}) has the following geometric interpretation:
each $A\in\G$ determines a Riemannian isometry
\begin{equation}\label{id}
(N_{A\cdot\mu},\ip)\longrightarrow (N_{\mu},\la A\cdot,A\cdot\ra)
\end{equation}
by exponentiating the Lie algebra isomorphism $A^{-1}:(\RR^n,A\cdot\mu)\longrightarrow(\RR^n,\mu)$.  It follows from (\ref{id}) that
\begin{quote}
the canonical basis $\{ e_1,\dots,e_n\}$ is stably Ricci-diagonal for $\ngo$ if and only the matrix of $\Ricci_{A\cdot\lb}$ is diagonal for any diagonal $A \in \G$.
\end{quote}

It is proved in \cite[Proposition 3.5]{minimal} that when $\mu$ is in addition nilpotent,
\begin{equation}\label{defmm}
\la \Ricci_\mu,\alpha\ra=4\la\pi(\alpha)\mu,\mu\ra, \qquad \forall\alpha\in\sym(n).
\end{equation}

\begin{remark}
This is equivalent to say that the
moment map $m:V\smallsetminus\{ 0\}\longrightarrow\sym(n)$ for the action (\ref{action}) is given by
$m(\mu)=\tfrac{4}{||\mu||^2}\Ricci_{\mu}$, a remarkable fact which is really in the core not only of the main result obtained in the present paper but of the whole subject of Ricci flow and solitons on nilmanifolds (see \cite{cruzchica,nilricciflow}).
\end{remark}

In the light of what has already been showed in this section, Theorem \ref{Thmintro} follows from the equivalence between parts (i) and (iv) in the following theorem.

\begin{theorem}\label{Thm} The following conditions are equivalent:
\begin{enumerate}
  \item[(i)] The canonical basis $\{ e_1,\dots,e_n\}$ is nice for $\ngo$.
  \item[(ii)] $\langle \pi(X)v_{ijk}, v_{rst}\rangle =0$, for all $X\in \ggo_\lambda,$ $\lambda \in \Phi,$ $c_{ij}^k, c_{rs}^t\ne 0.$
  \item[(iii)] $ \ricci_{A\cdot\mu}(e_l,e_m) =0$ for all $l \ne m$ and any diagonal $A \in \G$.
   \item[(iv)] The canonical basis $\{ e_1,\dots,e_n\}$ is stably Ricci-diagonal for $\ngo$.
\end{enumerate}
\end{theorem}

\begin{proof}
We first note that parts (iii) and (iv) are equivalent by (\ref{id}), as it was mentioned above.  By evaluating the maps $\pi([\alpha, X])=\pi(\alpha)\pi(X)-\pi(X)\pi(\alpha)$ at $v_{ijk}$ and then taking scalar product with $v_{rst}$ we obtain
$$
\langle\lambda+\alpha_{ij}^k,\alpha\rangle\langle\pi(X)v_{ijk},v_{rst}\rangle = \langle \alpha_{rs}^t,\alpha\rangle \langle \pi(X)v_{ijk},v_{rst}\rangle, \qquad \forall X\in \ggo_\lambda, \quad \lambda \in \Phi, \quad \alpha \in \ag,
$$
from which follows that
\begin{equation}\label{ecalfa}
\langle \pi(X)v_{ijk},v_{rst}\rangle  \ne 0 \quad\mbox{if and only if}\quad \alpha_{rs}^t - \alpha_{ij}^k \in \Phi.
\end{equation}
The equivalence between parts (i) and (ii) therefore follows from Lemma \ref{nicemu}.

By using (\ref{defmm}) and the fact that for any $A=\diag(a_1,\dots,a_n) \in \G$, $A\cdot  v_{ijk}= \tfrac{a_k}{a_ia_j}v_{ijk}$, we get that
\begin{align}
\ricci_{A\cdot\lb}(e_l,e_m) &= \langle \pi(E_{lm})A\cdot \lb, A\cdot \lb\rangle \notag \\
 &= \sum c_{ij}^k c_{rs}^t \langle \pi(E_{lm})A \cdot v_{ijk},A.v_{rst}\rangle  \label{diagonal} \\
 &= \sum c_{ij}^k c_{rs}^t \tfrac{a_k}{a_ia_j}\tfrac{a_t}{a_ra_s}\langle \pi(E_{lm})v_{ijk},v_{rst}\rangle, \notag
\end{align}
and thus part (iii) follows from (ii) since for all $l>m$, $E_{lm}\in \ggo_{\lambda}$ for $\lambda=E_{ll}-E_{mm}\in \Phi$.

Finally, we will show that part (iii) implies (ii). We begin by noting that if $A=\exp(\alpha),$ $\alpha \in \ag$, then
$$
\exp(\alpha)\cdot v_{ijk}=e^{\langle\alpha_{ij}^k,\alpha\rangle}v_{ijk}.
$$
It follows as in (\ref{diagonal}) and by using part (iii) that for all $l>m$,
\begin{equation}\label{alphacero}
0=\ricci_{A\cdot\lb}(e_l,e_m)= \sum c_{ij}^k c_{rs}^t e^{\langle\alpha_{ij}^k+\alpha_{rs}^t,\alpha\rangle}\langle \pi(E_{lm})v_{ijk},v_{rst}\rangle,
\end{equation}
where the sum runs only over the indices $ijk$ and $rst$ such that $\alpha_{ij}^k-\alpha_{rs}^t =E_{lm} \in \Phi$, since otherwise $\langle \pi(E_{lm})v_{ijk},v_{rst}\rangle =0$ by (\ref{ecalfa}).  If $\alpha_{i_oj_o}^{k_o}-\alpha_{r_os_o}^{t_o}=E_{lm}$, then there exists $\alpha_o \in \ag$ such that
$0\ne\langle\alpha_{i_oj_o}^{k_o}+\alpha_{r_os_o}^{t_o},\alpha_o\rangle$ and $\langle\alpha_{i_oj_o}^{k_o}+\alpha_{r_os_o}^{t_o},\alpha_o\rangle \ne \langle\alpha_{ij}^k+\alpha_{rs}^t,\alpha_o\rangle$ for any pair $ijk$, $rst$ of indices with $\alpha_{ij}^{k}-\alpha_{rs}^t =E_{lm}$ (note that $\alpha_{i_oj_o}^{k_o}+\alpha_{r_os_o}^{t_o}\ne \alpha_{ij}^k+\alpha_{rs}^t$ for any such a $3$-uple of indices).  It follows from (\ref{alphacero}) that
$$
0=\sum c_{ij}^k c_{rs}^t \langle \pi(E_{lm})v_{ijk},v_{rst}\rangle e^{t\langle\alpha_{ij}^k+\alpha_{rs}^t,\alpha_o\rangle}= \sum_{n=1}^N a_n e^{tb_n}, \qquad\forall t\in\RR,
$$
where in the last equality we have joined terms with the same exponent.  Thus $a_n =0$  for all $n$, and in particular, for the exponent $b_{n_o}= \langle\alpha_{i_oj_o}^{k_o}+\alpha_{r_os_o}^{t_o},\alpha_o\rangle$, whose coefficient is given by
$$
a_{n_o}=c_{i_oj_o}^{k_o}.c_{r_os_o}^{t_o}\la\pi(E_{lm})v_{i_oj_ok_o},v_{r_os_ot_o}\ra=0.
$$
This implies that $c_{i_oj_o}^{k_o}.c_{r_os_o}^{t_o}=0$, and so part (ii) follows, concluding the proof of the theorem.
\end{proof}

\section{Ricci flow on Lie groups}\label{RF}

Let $(N,g_0)$ be a (non-necessarily nilpotent) Lie group endowed with a left-invariant metric with metric Lie algebra $(\ngo,\ip_0)$.  Let $g(t)$ be a solution to the {\it Ricci flow}
\begin{equation}\label{RF}
\dpar g(t)=-2\ricci(g(t)),\qquad g(0)=g_0.
\end{equation}
The short time existence of a solution follows from \cite{Shi}, as $g_0$ is homogeneous and hence complete and of bounded curvature. Alternatively, one may require $N$-invariance of $g(t)$ for all $t$, and thus the metric Lie algebra of $(N,g(t))$ would have the form $(\ngo,\ip_t)$, where $\ip_t:=g(t)(e)$.  The Ricci flow equation (\ref{RF}) is therefore equivalent to the ODE
\begin{equation}\label{RFip}
\ddt \ip_t=-2\ricci(\ip_t),
\end{equation}
where $\ricci(\ip_t):=\ricci(g(t))(e):\ngo\times\ngo\longrightarrow\RR$, and hence short time existence and uniqueness of the solution in the class of $N$-invariant metrics is guaranteed.  In this way, $g(t)$ is homogeneous for all $t$, and hence the uniqueness within the set of complete and with bounded curvature metrics follows from \cite{ChnZhu}.  It is actually a simple matter to prove that such a uniqueness result, in turn, implies our assumption of $N$-invariance, as the solution must preserve isometries.  The need for this circular argument is due to the fact that the uniqueness of the Ricci flow solution is still an open problem in the noncompact general case (see \cite{Chn}).

In any case, there is an interval $(a,b)\subset\RR$ such that $0\in (a,b)$ and where existence and uniqueness (within complete and with bounded curvature metrics) of the Ricci flow $g(t)$ starting at $(N,g_0)$ hold.

\begin{remark}
One can use in the case of Lie groups the existence and uniqueness of the solution $g(t)$, forward and backward from any $t\in(a,b)$, to get that the isometry groups satisfy $\Iso(N,g(t))=\Iso(N,g_0)$ for all $t$.  This fact has recently been proved for the more general class of all complete and with bounded curvature metrics in \cite{Kts}.
\end{remark}

It is easy to prove that if $P(t)$ is the smooth curve of positive definite operators of $(\ngo,\ip_0)$ such that
$$
\ip_t=\la P(t)\cdot,\cdot\ra_0,
$$
then the Ricci flow equation (\ref{RFip}) determines the following ODE for $P(t)$:
\begin{equation}\label{RFP}
\ddt P(t)=-2P(t)\Ricci_t, \qquad P(0)=I,
\end{equation}
where $\Ricci_t:=\Ricci(g(t))(e):\ngo\longrightarrow\ngo$ is the Ricci operator.

In this section, we aim to understand under what conditions on the starting metric $(N,g_0)$ one has that the Ricci flow solution $g(t)$ is {\it diagonal}, in the sense that $P(t)$ is diagonal when written in some fixed orthonormal basis of $(\ngo,\ip_0)$ for all $t\in(a,b)$ (the metric $g_0$ will be called in such a case {\it Ricci flow diagonal}).  It is easy to check that this is equivalent to have the same property for $\Ricci_t$ (see (\ref{RFP})), and also to the commutativity of the family of symmetric operators $\{ P(t):t\in(a,b)\}$.

Clearly, the existence of an orthonormal stably-Ricci diagonal basis $\beta$ for $(\ngo,\ip_0)$ implies that $g(t)$ is diagonal, as the diagonal matrices with respect to $\beta$ are invariant under the ODE (\ref{RFP}).  It follows from Theorem \ref{Thmintro} that when $N$ is nilpotent, stably-Ricci diagonal bases must be necessarily nice, which shows that this is really a too strong condition to ask to a generic left-invariant metric, and even worst, to any left-invariant metric if the nilpotent Lie algebra $\ngo$ happens to have no nice basis whatsoever.

We now show, as an application of Theorem \ref{Thmintro}, that the existence of an orthonormal stably-Ricci diagonal basis for $(\ngo,\ip_0)$ is not necessary to get a diagonal Ricci flow solution $g(t)$.

\begin{example}\label{RS}
$(N,g_0)$ is said to be an {\it algebraic soliton} if $\Ricci_0=cI+D$ for some $c\in\RR$ and $D\in\Der(\ngo)$ (these metrics are called in the literature {\it nilsolitons} and {\it solvsolitons} in the case when $N$ is nilpotent or solvable, respectively).  When $N$ is simply connected, this condition implies that $(N,g_0)$ is indeed a Ricci soliton, as it easily follows that $\ricci(g_0)=cg_0-\unm L_{X_D}g_0$, where $X_D$ denotes the field determined by the one-parameter group of automorphisms of $N$ with derivatives $e^{tD}\in\Aut(\ngo)$.  It is easy to check that in this case,
$$
P(t)=(-2ct+1)e^{\tfrac{\log(-2ct+1)}{c}D} =e^{\tfrac{\log(-2ct+1)}{c}\Ricci_0}
$$
is the solution to (\ref{RFP}), and thus the Ricci flow $g(t)$ is always diagonal for algebraic solitons.  Indeed, if $\{ X_1,\dots,X_n\}$ is an orthonormal basis of eigenvectors of $\Ricci_0$ with respective eigenvalues $r_1,\dots,r_n$, then
$$
P(t)=\diag\left((-2ct+1)^{r_1/c},\dots,(-2ct+1)^{r_n/c}\right),
$$
which generalizes to the class of all algebraic solitons results on the asymptotic behavior of some nilsolitons obtained in \cite[Theorem 2.2]{Pyn} and \cite{Wllm}.
\end{example}

Any nilsoliton is therefore Ricci flow diagonal, though the Lie algebras of many of them do not admit any nice basis, and consequently do not admit stably-Ricci diagonal bases either by Theorem \ref{Thmintro}.  The lowest dimensional example of a nilsoliton whose nilpotent Lie algebra does not admit any nice basis is given in Proposition \ref{ex6} and has dimension $6$.  As another example of a nilsoliton which does not admit any nice basis we can take the free $3$-step nilpotent Lie algebra in $3$ generators (see \cite{Nkl3}).

\begin{remark}
We take this opportunity to point out that the explicit nilsoliton metric for the group in Proposition \ref{ex6} given in \cite[Example 3.4]{Wll} is wrong.  The formula for the Ricci operator of any diagonal metric is not always diagonal as it is asserted there.  Nevertheless, the existence of a nilsoliton metric (with the same eigenvalue-type) can be proved by using an approach similar to \cite[Example 3.3]{Frn}.
\end{remark}

In the non-nilpotent case, examples of diagonal Ricci flow solutions which are not written in a nice basis are much easier to find, as the following $3$-dimensional solvable example shows.

\begin{example}\label{solv3}
Let $\sg_3$ be the $3$-dimensional Lie algebra defined by
$$
[X_1,X_3]=X_2+X_3.
$$
The basis $\{ X_1,X_2,X_3\}$ is not nice but it is stably-Ricci diagonal.  Indeed, for any inner product $\ip$ such that
$$
\la X_i,X_i\ra=a_i^2, \quad a_i>0, \qquad \la X_i,X_j\ra=0 \qquad\forall i\ne j,
$$
we use (\ref{ricci2}) to get that $H=\tfrac{1}{a_1^2}X_1$ and thus for all $r\ne s$ the corresponding Ricci tensor satisfies
\begin{align*}
\ricci(X_r,X_s) &=\unc\sum\la [\tfrac{1}{a_i}X_i,\tfrac{1}{a_j}X_j],X_r\ra \la [\tfrac{1}{a_i}X_i,\tfrac{1}{a_j}X_j],X_s\ra \\ 
& -\unm\la [H,X_r],X_s\ra-\unm\la X_r,[H,X_s]\ra.
\end{align*}
This clearly vanishes if either $r$ or $s$ is $1$, and
\begin{align*}
\ricci(X_2,X_3) & =\unm\la [\tfrac{1}{a_1}X_1,\tfrac{1}{a_3}X_3],X_2\ra \la [\tfrac{1}{a_1}X_1,\tfrac{1}{a_3}X_3],X_3\ra -\unm\la X_2,[\tfrac{1}{a_1^2}X_1,X_3]\ra \\
&=\unm\left(\tfrac{a_2}{a_1}\right)^2- \unm\left(\tfrac{a_2}{a_1}\right)^2=0.
\end{align*}
This example also shows that, beyond nilpotent Lie groups, Theorem \ref{Thmintro} is not longer true.
\end{example}

On the other hand, in the non-nilpotent case, it is also possible to find nice basis which are not stably-Ricci diagonal, as the following two examples show.  In the first one, what fails to be diagonal is the `unimodularity' part $S(\ad{H})$ in formula (\ref{ricci2}), while in the second one the Killing form part $B$ is not diagonal.

\begin{example}\label{solv4}
Let $\sg_4$ be the solvable Lie algebra defined by
$$
[X_1,X_2]=2X_3, \qquad [X_1,X_3]=X_2, \qquad [X_1,X_4]=X_4.
$$
If $\ip$ is the inner product on $\sg_4$ for which the basis $\{ X_1,\dots,X_4\}$ is orthonormal, then the Ricci operator is
$$
\Ricci_{\ip}=\left[\begin{array}{cccc}
-\tfrac{11}{2}&0&0&0 \\
0&-\tfrac{3}{2}&-\tfrac{3}{2}&0 \\
0&-\tfrac{3}{2}&\tfrac{3}{2}&0 \\
0&0&0&-1
\end{array}\right].
$$
Notice that the basis is however nice.
\end{example}

\begin{example}\label{sl3}
Let $\{ X_1,X_2,X_3\}$ be a basis of $\slg_2(\RR)$ such that
$$
[X_1,X_2]=X_2, \qquad [X_1,X_3]=-X_3, \qquad [X_2,X_3]=X_1.
$$
This basis is nice but the Ricci operator of the metric which makes it orthonormal equals
$$
\Ricci=\left[\begin{array}{ccc}
-\tfrac{3}{2} &0&0 \\
0& -1&-1 \\
0&-1&-\unm
\end{array}\right].
$$
\end{example}

We finally provide an example of a left-invariant metric on a $4$-dimensional nilpotent Lie group whose Ricci flow is not diagonal.  Some other left-invariant metrics on this group were proved to be Ricci flow diagonal in \cite[Proposition 3.5]{IsnJckLu}.

\begin{example}\label{nodiag}
Let $\ngo_4$ be the $4$-dimensional $3$-step nilpotent Lie algebra defined by
$$
[X_1,X_2]=\sqrt{2}X_3+\sqrt{2}X_4, \qquad [X_1,X_3]=\sqrt{2}X_4.
$$
If $\ip_0$ is the inner product on $\ngo_4$ for which the basis $\{ X_1,\dots,X_4\}$ is orthonormal, then the Ricci operator of the corresponding nilmanifold $(N_4,g_0)$, written in this basis, is given by
$$
\Ricci_0=\left[\begin{array}{cccc}
-3&0&0&0 \\
0&-2&-1&0 \\
0&-1&0&1 \\
0&0&1&2
\end{array}\right].
$$
The following is an orthonormal basis of eigenvectors of $\Ricci_0$, which is unique up to $\pm 1$-scaling and permutations,
$$
\begin{array}{lcl}
Y_1=(1,0,0,0), && Y_2=\left(0,-\unm-\tfrac{1}{\sqrt{6}},-\tfrac{1}{\sqrt{6}},\unm-\tfrac{1}{\sqrt{6}}\right), \\
Y_3=\tfrac{1}{\sqrt{6}}(0,1,-2,1), && Y_4=\left(0,-\unm+\tfrac{1}{\sqrt{6}}, \tfrac{1}{\sqrt{6}}, \unm+\tfrac{1}{\sqrt{6}}\right),
\end{array}
$$
with respective eigenvalues $-3,-\sqrt{6},0,\sqrt{6}$.  Let us assume that the Ricci flow $g(t)$ starting at $(N_4,g_0)$ is diagonal.  Thus there exists an orthonormal basis of $(\ngo_4,\ip_0)$ such that the corresponding matrix of $\Ricci_t$ is diagonal for all $t$; in particular, for $t=0$, and so such a basis must be $\{ Y_1,\dots,Y_4\}$ up to $\pm 1$-scaling and permutations.  This implies that both $\Ricci_t$ and the solution $P(t)$ to the ODE (\ref{RFP}) are diagonal with respect to the basis $\{Y_1,\dots,Y_4\}$ for all $t$, say $P(t)=\diag(a^{-2},b^{-2},c^{-2},d^{-2})$, $a,b,c,d>0$.  A straightforward computation gives that the Ricci curvature therefore satisfies
\begin{align*}
  \la\Ricci_tY_2,Y_3\ra_t &= \tfrac{a^2}{(\sqrt{6}+2)^2}\left(\tfrac{4+\sqrt{6}}{24}\tfrac{d^2}{bc}+\tfrac{12+5\sqrt{6}}{8}\tfrac{c}{b}+\tfrac{12 +5\sqrt{6}}{8}\tfrac{b}{c}-\tfrac{76+31\sqrt{6}}{24}\tfrac{bc}{d^2}\right) =0,\\
  \la\Ricci_tY_2,Y_4\ra_t &= \tfrac{a^2}{(\sqrt{6}+2)^2}\left(\tfrac{4+\sqrt{6}}{24}\tfrac{bd}{c^2}+\tfrac{12+5\sqrt{6}}{8}\tfrac{d}{b}+\tfrac{12 +5\sqrt{6}}{8}\tfrac{b}{d}-\tfrac{76+31\sqrt{6}}{24}\tfrac{c^2}{bd}\right) =0, \\
  \la\Ricci_tY_3,Y_4\ra_t &= \tfrac{a^2(5+2\sqrt{6})}{12(\sqrt{6}+2)^2b^2cd} (b^4-c^2d^2) =0. 
\end{align*}
All positive solutions to this system are easily seen to satisfy  $b^2=c^2=d^2$.  This in turn implies that
$$
\la\Ricci_tY_2,Y_2\ra=0, \qquad \la\Ricci_tY_3,Y_3\ra= \tfrac{a^2(12+5\sqrt{6})}{(\sqrt{6}+2)^2},
$$
and hence $b^2$ and $c^2$ can not solve the ODE system determined by (\ref{RFP}) and still be equal each other, a contradiction.  We conclude that the Ricci flow $g(t)$ starting at the nilmanifold $(N_4,g_0)$ is not diagonal with respect to any orthonormal basis.
\end{example}

\section{Appendix: Ricci curvature of left invariant metrics}

We give in this section a formula for the Ricci operator of a left-invariant metric on a Lie group with corresponding metric Lie algebra $(\ggo,\ip)$.  There exists a unique element $H\in\ggo$ such that $\la
H,X\ra=\tr{\ad{X}}$ for any $X\in\ggo$.  If $B$ denotes the symmetric operator defined by
the Killing form of $\ggo$ relative to $\ip$ (i.e. $\la BX,Y\ra=\tr{\ad{X}\ad{Y}}$ for
all $X,Y\in\ggo$), then the Ricci operator
$\Ricci$ of $(\ggo,\ip)$ is given by (see for instance \cite[7.38]{Bss}):
\begin{equation}\label{ricci2}
\Ricci=M-\unm B-S(\ad{H}),
\end{equation}
where $S(\ad{H})=\unm(\ad{H}+(\ad{H})^t)$ is the symmetric part of
$\ad{H}$ and $M$ is the symmetric operator such that $\la MX,Y\ra$ equals the right hand side of formula (\ref{ricci}) for
all $X,Y\in\ggo$, where $\{ X_i\}$ can be any orthonormal basis of $(\ggo,\ip)$.

\end{document}